\documentclass{amsart}

\usepackage{graphicx}
\usepackage{amsopn}
\usepackage{amsmath}
\usepackage{amsfonts}
\usepackage{amssymb}
\usepackage{amscd}
\usepackage{mathtools}
\usepackage{mathrsfs}
\usepackage{url}
\usepackage{multirow,multicol}
\usepackage{cleveref,xcolor}
\usepackage{booktabs}

\newcommand{\mA}{\mathcal{A}}
\newcommand{\supp}[1]{\operatorname{supp}(#1)}
\def\rank{\mbox{rank}}
\newcommand{\re}{\mathbb{R}}

\newcommand{\N}{\mathbb{N}}

\newcommand{\lmd}{\lambda}

\newcommand{\Dt}{ {\rm \Delta} }
\newcommand{\Sig}{ {\rm \Sigma} }
\newcommand{\st}{\mathrm{s.t.}}
\newcommand{\reff}[1]{(\ref{#1})}
\def\af{\alpha}

\def\gm{\gamma}
\newcommand{\mc}[1]{\mathcal{#1}}
\newcommand{\idl}{\mbox{Ideal}}
\newcommand{\qmod}{\mbox{Qmod}}

\newcommand{\be}{\begin{equation}}
	\newcommand{\ee}{\end{equation}}
\newcommand{\beq}{\begin{equation}}
	\newcommand{\eeq}{\end{equation}}
\newcommand{\baray}{\begin{array}}
	\newcommand{\earay}{\end{array}}

\newcommand{\bbm}{\begin{bmatrix}}
	\newcommand{\ebm}{\end{bmatrix}}

\newcommand{\bit}{\begin{itemize}}
	\newcommand{\eit}{\end{itemize}}

\newcommand{\bdes}{\begin{description}}
	\newcommand{\edes}{\end{description}}

\makeatletter
\def\@thmcountersep{.}
\makeatother

\newtheorem{Theorem}{Theorem}[section]
\newtheorem{thm}[Theorem]{Theorem}

\newtheorem{lem}[Theorem]{Lemma}

\theoremstyle{definition}
\newtheorem{defi}[Theorem]{Definition}
\newtheorem{Example}[Theorem]{Example}

\theoremstyle{remark}
\newtheorem{Remark*}[Theorem]{Remark}
\newtheorem{rmk}[Theorem]{Remark}

\numberwithin{equation}{section}

\begin{document}

\title{Sparse Copositive Polynomial Optimization}

% --- Standardized Author Information for amsart ---
\author{Suhan Zhong}
\address{School of Mathematical Sciences, Shanghai Jiaotong University, Shanghai, 200240, China}
\email{suzhong@sjtu.edu.cn}

\author{Jinling Zhou}
\address{School of Mathematics and Computational Science, Xiangtan University, Xiangtan, Hunan, 411105, China}
\email{jinlingzhou@smail.xtu.edu.cn}

\author{Jiawang Nie}
\address{Department of Mathematics, University of California San Diego, 9500 Gilman Drive, La Jolla, CA, USA, 92093}
\email{njw@math.ucsd.edu}

\author{Xindong Tang}
\address{Department of Mathematics, Hong Kong Baptist University, Kowloon Tong, Kowloon, Hong Kong}
\email{xdtang@hkbu.edu.hk}

% In amsart, abstract and keywords must come before \maketitle
\begin{abstract}
This paper studies the copositive optimization problem whose objective is a sparse polynomial, with linear constraints over the nonnegative orthant. We propose sparse Moment-SOS relaxations to solve it. Necessary and sufficient conditions are shown for these relaxations to be tight. In particular, we prove they are tight under the cop-SOS convexity assumption. Compared to the traditional dense ones, the sparse Moment-SOS relaxations are more computationally efficient. Numerical experiments are given to show the efficiency.
\end{abstract}

\keywords{copositive polynomial, sparse, cop-SOS convex, moment-SOS, relaxation}
\subjclass[2020]{90C23, 90C22, 90C30}

\maketitle

\section{Introduction}

Let $x \coloneqq (x_1,\ldots, x_n)$ be a vector variable
in the $n$-dimensional real Euclidean space $\re^n$.
We consider the sparse polynomial optimization problem
\begin{equation}\label{eq:spa_pop}
\left\{\begin{array}{cl}
	\min\limits_{x\in\re^n} &
	f(x) \coloneqq f_1(x_{\Dt_1}) + \cdots + f_m(x_{\Dt_m})\\
	\st & Ax = b,\, Cx\ge d,\, x\ge 0, 
\end{array}
\right.
\end{equation}
for given matrices $A\in\re^{m_1\times n}$, $C\in\re^{m_2\times n}$
and vectors $b\in\re^{m_1}$, $d\in\re^{m_2}$.
In the above, $\Dt_1,\ldots, \Dt_m$ are subsets of $[n]\coloneqq \{1,\ldots, n\}$
such that
\[ \Dt_1\cup\cdots\cup \Dt_m  \, =  \,  [n] .\]
Each $f_i(x_{\Dt_i})$ is a polynomial in the subvector
\[
x_{\Dt_i}\coloneqq (x_{j_1}, \cdots, x_{j_{n_i}}),
\qquad \text{where} \quad 
\Dt_i = \{j_1,\ldots, j_{n_i}\} .
\]
A polynomial like $f$ as in \reff{eq:spa_pop} is said to have 
the sparsity pattern $(\Dt_1,\ldots, \Dt_m)$. 
In the above, the matrices $A$ and $C$ can be either dense or sparse.
The minimum value of \reff{eq:spa_pop} is denoted as $f_{min}$,
and its feasible region is 
\be \label{eq:setK}
K = \{ x\in\re^n: Ax=b,\, Cx\ge d,\, x\ge 0 \}.
\ee
We call \reff{eq:spa_pop} a {\it sparse copositive polynomial optimization} problem. This is not the same as {\it copositive programming},
which typically means linear conic optimization 
problems over the copositive matrix cone
\cite{AhmedDurStill13,Burer15,DobreVera15,Dur10}.

Problem~\eqref{eq:spa_pop} can be viewed as a general dense polynomial optimization problem. It can be solved by the standard Moment-SOS hierarchy of semidefinite programming (SDP) relaxations 
\cite{Las01,LasBk15,Lau09,NieBook}.
We call these relaxations to be the {\it dense} ones. 
This hierarchy has asymptotic convergence under the Archimedeanness \cite{Las01} and has finite convergence under some additional optimality conditions \cite{nieopcd}. It is typically expensive to solve the dense Moment-SOS relaxations for large-scale problems.

Recently, sparse Moment-SOS relaxations have been extensively used
to solve large-scale polynomial optimization problems. In them, the objective function usually has the sparsity pattern like $f$ as in (\ref{eq:spa_pop}), 
and each constraining polynomial is only about a subvector $x_{\Dt_i}$.
By exploiting the sparsity, the sparse relaxation formulates semidefinite programs whose positive semidefinite matrix variables have smaller dimensions, compared with the dense one. Sparse relaxations are efficient for solving large scale problems. They have asymptotic convergence under certain conditions. We refer to the earlier work \cite{korda2025convergence,Lasspacov06,vmjw21,NieDemmel09,CSTSSOS,waki2006sums}.
The finite convergence property of sparse Moment-SOS relaxations is investigated in \cite{NieQuTangZhang25,NQTZ25mat}.
Moreover, there exist other types of sparsity patterns in polynomial optimization; see \cite{Kim2011,KLMS23,MagronWang23,TSSOS}.
We remark that problem~\eqref{eq:spa_pop} is different from other sparse polynomial optimization problems, since its linear constraints can be dense. The prior sparse relaxation methods are not appropriate for solving the sparse copositive polynomial optimization problem (\ref{eq:spa_pop}). In this paper, we propose new sparse Moment-SOS relaxations to solve \eqref{eq:spa_pop}.

Copositive polynomials have been extensively and actively studied in
optimization \cite{BomzeDur00,Dur10,NieYangZhang18,NieDehom23}.
A homogeneous polynomial $p(x)$ of degree $d$
is said to be {\it copositive} if $p(x)\ge 0$ for all $x \ge 0$.
For the degree two case, a symmetric matrix $A$ is copositive
if and only if the associated quadratic form is copositive.
The set of all $d$th degree copositive forms in $n$ variables is a convex cone, which we denote by $\mc{COP}_{n,d}$.
Copositive polynomials and their applications are studied in \cite{EichfelderJahn08,EichfelderPovh13,Lasserre14,NieYangZhang18}.
Checking memberships of the copositive cone $\mc{COP}_{n,d}$
is NP-hard \cite{DickinsonGijben14}. We refer to
\cite{HuangXie25,NieYangZhang18,NieDehom23,PowersReznick01}
for the related work.
In the existing literature (e.g., \cite{AhmedDurStill13,Burer15,DobreVera15,Dur10}),
the term {\it copositive programming} or {\it copositive optimization} usually refers to the linear conic optimization problem over the copositive matrix cone $\mc{COP}_{n,2}$. We refer to
\cite{Burer09,Dukanovic10,PovhRendl07,VargasLaurent23}
for applications of copositive optimization.

\medskip \noindent
{\bf Contributions.}
Problem~\eqref{eq:spa_pop} is a class of linearly constrained sparse polynomial optimization problems. 
It can be equivalently formulated as  
\be \label{sparcop:conic}
\left\{\begin{array}{cl}
	\max  & \gamma   \\
	\st & f(x) -\gamma\in \mathscr{P}_d(K), 
\end{array}
\right.
\ee
where $d$ is the degree of $f$ and $\mathscr{P}_d(K)$ denotes the cone of degree-$d$ 
polynomials that are nonnegative on $K$, i.e.,   
\be \label{cone:Cnd}
\mathscr{P}_d(K)  \coloneqq \{ q\in \re[x]_d: q(x)\ge 0 \,\,\forall x\in K\}.
\ee
In this paper, we propose new sparse Moment-SOS relaxations to solve \reff{sparcop:conic}, by exploiting the sparsity pattern of $f$, and  investigate conditions for these relaxations to be tight.
The sparse relaxations are more appropriate for solving large-scale problems. Our main contributions are:

\begin{itemize}
	
\item 
We propose new sparse Moment-SOS relaxations to solve \eqref{eq:spa_pop}, or its equivalent reformulation \reff{sparcop:conic}.  The sparse moment and SOS relaxations are given in \reff{eq:spcopmom} and \reff{eq:spcopsos}, respectively.

\item We characterize conditions for the sparse Moment-SOS relaxations \reff{eq:spcopmom}-\reff{eq:spcopsos} to be tight. 
This is shown in Theorems~\ref{thm:tight:char}, \ref{thm:rank1_tight}, and \ref{thm:multipoints}.

\item We show that the sparse Moment-SOS relaxations \reff{eq:spcopmom}-\reff{eq:spcopsos} are tight under 
the cop-SOS convexity assumption.
This is shown in Theorem~\ref{thm:cop-sos-cvx}.
	
\end{itemize}

The rest of this paper is organized as follows. In Section~\ref{sec:pre}, we introduce some notation and give preliminaries for polynomial optimization. In Section~\ref{sec:spmomsos}, we propose the sparse Moment-SOS relaxations \reff{eq:spcopmom}-\reff{eq:spcopsos} and characterize conditions for their tightness. In Section~\ref{sec:tightness}, we prove these sparse relaxations are tight under the cop-SOS convexity assumption. Numerical experiments are presented in Section~\ref{sec:num} and some conclusions are drawn in Section~\ref{sec:con}.

\section{Preliminaries}\label{sec:pre}

\noindent 
{\it Notation} \, 
The symbol $\re$ denotes the real field,
$\re_+$ denotes the set of nonnegative real numbers,
and $\N$ denotes the set of nonnegative integers.
For a positive real number $t$, $\lceil t\rceil$ denotes
the smallest integer that is greater than or equal to $t$.
For a positive integer $n$, denote $[n]\coloneqq \{1,\ldots, n\}$.
The notation $\re^n$ (resp. $\re_+^n$, $\N^n$) stands for the set of
$n$-dimensional vectors of entries in $\re$ (resp. $\re_+$, $\N$).
We denote the subspace
\[ 
	\re^{\Dt_i} \coloneqq \{x = (x_1,\ldots, x_n)\in\re^n: 
	x_j = 0\,\,\text{for all}\,\, j\not\in \Dt_i\} ,
\]
and denote $\re_+^{\Dt_i}$ and $\N^{\Dt_i}$ similarly.
The notation $\re[x]$ and $\re[x_{\Dt_i}]$ stand for the polynomial rings in
$x$ and $x_{\Dt_i}$, respectively.
The set of polynomials in $\re[x]$ with degrees up to $d$
is denoted as $\re[x]_d$, and $\re[x_{\Dt_i}]_d$ is defined similarly.
The column vector of monomials in $x$ and of degrees up to $d$,
ordered in the graded lexicographic ordering, is denoted by
\[
[x]_d  \coloneqq \big[ \baray{ccccccccccccc}
1 &  x_1  & \cdots  &  x_n   & x_1^2  &  x_1 x_2 &  \cdots  &   x_n^2  &    \cdots \,  &
x_1^d  &   x_1^{d-1}x_2  &  \cdots  & x_n^d
\earay]^T.
\]
The monomial vector $[x_{\Dt_i}]_d$ is defined similarly.

\subsection{Standard Moment-SOS relaxations}
\label{ssec:stan_mom_sos}

General polynomial optimization problems can be solved globally
by the dense Moment-SOS relaxations \cite{Las01}. 
A polynomial $\sigma\in\re[x]$ is said to be a sum-of-squares (SOS) if there exist polynomials 
$p_1,\ldots, p_l\in\re[x]$ such that $\sigma = p_1^2+\cdots+p_l^2$.
We denote by $\Sig[x]$ the cone of SOS polynomials in $x$.
For a degree $d$, we denote the truncation 
\[
\Sig[x]_d  \, \coloneqq \, \Sig[x]\cap \re[x]_d .
\]
SOS polynomials can be represented by semidefinite programming \cite{NieBook}. 
In particular, we have
$p \in \Sig[x]_{2d}$ if and only if there exists
a positive semidefinite matrix $P$ such that 
\[
p(x) \, = \, [x]_d^T P [x]_d .
\]

We consider the general polynomial optimization problem
\be \label{eq:pop}
\left\{\begin{array}{cl}
	\min  & f(x)\\
	\st & g(x) \coloneqq (g_1(x),\ldots, g_{m_1}(x)) = 0,\\
	&  h(x) \coloneqq 
	(h_1(x),\ldots, h_{m_2}(x))\ge 0.
\end{array}
\right.
\ee
Let $S$ be the feasible set of \reff{eq:pop}.
Then the above  is equivalent to
\[
\left\{\begin{array}{cl}
	\max  & \,\, \gamma \\
	\st & \,\,  f(x)-\gamma\ge 0 \,\, \text{on} \,\, S.
\end{array}
\right.
\]
To represent polynomials that are nonnegative on $S$,
we denote the ideal generated by $g$ as
\[
\idl[g]  \coloneqq  g_1\re[x]+\cdots+g_{m_1}\re[x],
\]
and denote the quadratic module generated by $h$ as
\[
\qmod[h]  \coloneqq \Sig[x] + h_1\Sig[x]+\cdots+h_{m_2}\Sig[x].
\]
Every polynomial $p\in\idl[g]+\qmod[h]$ is nonnegative on $S$,
but the reverse is not necessarily true.
The set $\idl[g]+\qmod[h]$ is said to be {\it Archimedean} 
if there exists $q \in \idl[g]+\qmod[h]$ such that 
the single inequality $q(x)\ge 0$ defines a compact set in $\re^n$.
When $\idl[g]+\qmod[h]$ is Archimedean, every polynomial that is positive
on $S$ belongs to $\idl[g]+\qmod[h]$. This result is referenced as Putinar's Positivstellensatz \cite{Putinar93}.

For a degree $k\in\N$ with $2k\ge \max\{\deg(g), \deg(h)\}$,
we denote the truncation
\[
\begin{array}{rcl}
\idl[g]_{2k} & \coloneqq & g_1\re[x]_{2k-\deg(g_1)}+\cdots+
g_{m_1}\re[x]_{2k-\deg(g_{m_1})},\\
\qmod[h]_{2k} & \coloneqq &  \Sig[x]_{2k}+ h_1\Sig[x]_{2k-\deg(h_1)}+
\cdots + h_{m_2}\Sig[x]_{2k-\deg(h_{m_2})}.
\end{array}
\]
The $k$th order dense SOS relaxation for solving \eqref{eq:pop} is
\be   \label{eq:pop_sos_k}
\left\{\begin{array}{cl}
	\max\limits_{\gamma\in\re} & \gamma\\
	\st & f(x)-\gamma\in \idl[g]_{2k}+\qmod[h]_{2k}.
\end{array}
\right.
\ee
This is equivalent to a semidefinite program.
Its dual problem corresponds to the $k$th order dense moment relaxation of \eqref{eq:pop}.
They are called the $k$th order Moment-SOS relaxations. 
When $\idl[g]+\qmod[h]$ is Archimedean, 
the optimal value of \eqref{eq:pop_sos_k} converges to that of \eqref{eq:pop} as $k$ goes to infinity \cite{Las01}. 
This is called the asymptotic convergence.
Additionally, if some optimality conditions hold for \eqref{eq:pop}, 
these two optimal values are equal when $k$ is large enough \cite{nieopcd}.
This is called the finite convergence.
The tightness of Moment-SOS relaxations can be checked by the {\it flat truncation} condition \cite{NieFlat13}.
We refer to \cite{NieBook} for more detailed introductions to Moment-SOS relaxations for polynomial optimization.

\subsection{Sparse polynomials and moments}
Let $\Dt_1,\ldots, \Dt_m$ be subsets of $[n]$ such that
$\Dt_1\cup\cdots\cup \Dt_m = [n]$.
Write $\Dt_i = \{j_1,\ldots, j_{n_i}\}$ for each $i=1,\ldots, m$.
The cone of all SOS polynomials in $x_{\Dt_i}$ is denoted as
$\Sig[x_{\Dt_i}]$. For a degree $d$, we denote
$\Sig[x_{\Dt_i}]_d\coloneqq \Sig[x_{\Dt_i}]\cap \re[x_{\Dt_i}]_d$.
The vector inequality $x\ge 0$ is equivalent to
$(x_{\Dt_1},\ldots, x_{\Dt_m})\ge 0$.
For $x_{\Dt_i}\coloneqq (x_{j_1},\ldots, x_{j_{n_i}})$
and a positive integer $k$, denote  
\begin{eqnarray*}
\qmod_{\Dt_i}[x_{\Dt_i}] &\coloneqq& \Sig[x_{\Dt_i}] + x_{j_1}\Sig[x_{\Dt_i}]+\cdots+ x_{j_{n_i}}\Sig[x_{\Dt_i}], \\
\qmod_{\Dt_i}[x_{\Dt_i}]_{2k} &\coloneqq& \Sig[x_{\Dt_i}]_{2k} + x_{j_1} \Sig[x_{\Dt_i}]_{2k-2}
	+\cdots+ x_{j_{n_i}}\Sig[x_{\Dt_i}]_{2k-2}.
\end{eqnarray*}
The $\qmod_{\Dt_i}[x_{\Dt_i}]$ is the sparse quadratic module of $x_{\Dt_i}$,
and $\qmod_{\Dt_i}[x_{\Dt_i}]_{2k}$ is the $k$th order truncation.
For convenience, we denote 
\begin{eqnarray*}
	\qmod[x]_{spa} &\coloneqq& \qmod_{\Dt_1}[x_{\Dt_1}] + \cdots + \qmod_{\Dt_m}[x_{\Dt_m}],\\
	\qmod[x]_{spa,2k} &\coloneqq& \qmod_{\Dt_1}[x_{\Dt_1}]_{2k} + \cdots + \qmod_{\Dt_m}[x_{\Dt_m}]_{2k}.
\end{eqnarray*} 

For a degree $d$, recall that the monomial vector of $x_{\Dt_i}$ with degrees up to $d$ is denoted as 
$[x_{\Dt_i}]_d \coloneqq (x^{\alpha})_{\alpha\in \N_d^{\Dt_i}}$,
where 
\[
\N_d^{\Dt_i}  \coloneqq \{\alpha\in \N^{\Dt_i}: \alpha_{j_1}+\cdots+\alpha_{j_{n_i}}\le d\} .
\]
For an even degree $2k$, let $\re^{\N_{2k}^{\Dt_i}}$ 
denote the space of all real vectors indexed by power vectors $\alpha\in \N_{2k}^{\Dt_i}$. 
Given  $y_{\Dt_i} \in\re^{\N_{2k}^{\Dt_i}}$, 
for a degree $k_0 \le k$, we define the truncation 
\be \label{trun:2k0}
y_{\Dt_i}|_{2k_0} \, \coloneqq \, 
(y_{\Dt_i})_{\alpha \in \N_{2k_0}^{\Dt_i} } .
\ee
A vector $y_{\Dt_i}\in\re^{\N_{2k}^{\Dt_i}}$ naturally defines the Riesz functional on $\re[x_{\Dt_i}]_{2k}$ such that
\[
\mathscr{R}_{y_{\Dt_i}}(x_{\Dt_i}^\alpha) \coloneqq (y_{\Dt_i})_{\alpha}
\quad \mbox{for}\quad \alpha\in \N_{2k}^{\Dt_i}.
\]
This gives the bilinear operation between $\re[x_{\Dt_i}]_{2k}$ and $\re^{\N_{2k}^{\Dt_i}}$:
\be \label{eq:bilinear}
\langle p, y_{\Dt_i}\rangle \, \coloneqq \, \mathscr{R}_{y_{\Dt_i}}(p).
\ee
For $p\in \re[x_{\Dt_i}]$ and $y_{\Dt_i}\in \re^{\N_{2k}^{\Dt_i}}$,
with $2k\ge \deg(p)$, the $k$th order localizing matrix associated with $p$
and $y_{\Dt_i}$ is
\[
L_p^{\Dt_i,k}[y_{\Dt_i}]  \, \coloneqq  \,
\mathscr{R}_{y_{\Dt_i}}
\big( p(x_{\Dt_i})[x_{\Dt_i}]_{k_1}[x_{\Dt_i}]_{k_1}^T \big),
\]
where $k_1=\lceil k-\deg(p)/2\rceil$.
In particular, when $p=1$ is the constant one polynomial,
the corresponding localizing matrix is called the $k$th order moment matrix
\[
M_{\Dt_i}^{(k)}[y_{\Dt_i}] \coloneqq L_1^{\Dt_i, k}[y_{\Dt_i}].
\]
For instance, when $n=4$, $\Dt_1 = \{1,2\}$, $\Dt_2 = \{2,3,4\}$, we have
\[
\begin{array}{l}
\N_2^{\Dt_1} = \{(0,0,0,0),\, (1,0,0,0),\,(0,1,0,0),\,
(2,0,0,0),\,(1,1,0,0),\, (0,2,0,0)\},\\
\N_2^{\Dt_2} = \{
(0,0,0,0),\, (0,1,0,0),\, (0,0,1,0),\,(0,0,0,1),\,
(0,2,0,0),\,(0,1,1,0),\\
\quad \quad \quad \,\,(0,1,0,1)\,(0,0,2,0),\,(0,0,1,1),\,(0,0,0,2) \},\\
\N_2^{\Dt_1}\cap \N_2^{\Dt_2} =
\{(0,0,0,0),\, (0,1,0,0),\,(0,2,0,0)\}, 
\end{array}
\]
\[
M_{\Dt_1}^{(1)}[y_{\Dt_1}] = \bbm
y_{0000} & y_{1000} & y_{0100} \\
y_{1000} & y_{2000} & y_{1100}\\
y_{0100} & y_{1100} & y_{0200}\\
\ebm, \,\,
L_{x_2}^{\Dt_2,2}[y_{\Dt_2}] = \bbm
y_{0100} & y_{0200} & y_{0110} & y_{0101}\\
y_{0200} & y_{0300} & y_{0210} & y_{0201}\\
y_{0110} & y_{0210} & y_{0120} & y_{0111}\\
y_{0101} & y_{0201} & y_{0111} & y_{0102}
\ebm.
\]
For a tuple of sparse polynomials 
$g_i = (g_{i,1},\ldots, g_{i,l_i})$ in $x_{\Dt_i}$, we denote
\[
L_{g_i}^{\Dt_i, k}[y_{\Dt_i}] \coloneqq
\operatorname{diag}\big(
L_{g_{i,1}}^{\Dt_i,k}[y_{\Dt_i}],\ldots, L_{g_{i,l_i}}^{\Dt_i,k}[y_{\Dt_i}]
\big),
\]
where $\operatorname{diag}(\cdot)$ refers to the operation of block diagonal matrix. We remark that if $M_{\Dt_i}^{(k)}[y_{\Dt_i}] \succeq 0$ and 
$L_{x_{\Dt_i}}^{\Dt_i, k}[y_{\Dt_i}] \succeq 0$, then 
\be \label{dual:QM:Dti}
\langle p, y_{\Dt_i} \rangle \,\, \ge  \,\, 0 
\quad \text{for all} \quad 
p \in \qmod_{\Dt_i}[x_{\Dt_i}]_{2k}.
\ee
We refer to \cite[Sec.~2.5]{NieBook} for this fact.

\section{The sparse Moment-SOS relaxations}
\label{sec:spmomsos}

In this section, we propose sparse Moment-SOS relaxations to solve the
sparse copositive optimization problem \eqref{eq:spa_pop}.
Let
\[
d \coloneqq \deg(f), \quad 
k_0 \coloneqq \lceil d/2\rceil.
\]
For an integer $k\ge k_0$, denote the label set
\[
\mathbb{U}_k \coloneqq \N_{2k}^{\Dt_1}\cup\cdots\cup \N_{2k}^{\Dt_m}.
\]
Since $\cup_{i=1}^m \Dt_i = [n]$, the set of canonical basis vectors $\{e_1,\ldots, e_n\}$ in $\N^n$
is a proper subset of $\mathbb{U}_k$.
We let $\re^{\mathbb{U}_k}$ denote the space of vectors $y$
that are labeled by $\af \in \mathbb{U}_k$.
Such $y$ is called a sparse truncated multi-sequence (tms).
Denote by $y_{\Dt_i}$ the subvector of $y$ that is indexed by  
$\af \in \N_{2k}^{\Dt_i}$, i.e., 
\[
y_{\Dt_i} \, \, = \, \, (y_\af)_{ \af \in \N_{2k}^{\Dt_i} }. 
\]
Define the projection map:
\be \label{map:pi}
\pi: \re^{\mathbb{U}_k}\to \re^n, \quad 
y \mapsto (y_{e_1},\ldots, y_{e_n}).
\ee
For Lagrange multiplier variables $\lmd \in \re^{m_1}$ and $\mu \in \re_+^{m_2}$, we denote the polynomial
\be \label{eq:g}
g(x;\lambda, \mu) \, \coloneqq \, \lambda^T(Ax-b)+\mu^T(Cx-d).
\ee
Since $\cup_{i=1}^m \Dt_i = [n]$, $x \ge 0$ is equivalent to $x_{\Dt_i} \ge 0$ for every $i$. To solve \eqref{eq:spa_pop}, for a relaxation order $k$,
we propose the sparse moment relaxation:
\be \label{eq:spcopmom}
\left\{\begin{array}{cl}
\min &
\langle f, y\rangle  \coloneqq  \langle f_1, y_{\Dt_1}\rangle +
\cdots + \langle f_m, y_{\Dt_m}\rangle \\
\st	& A\pi(y) = b,\,C\pi(y) \ge d,  \\ 
& M_{\Dt_i}^{(k)}[y_{\Dt_i}] \succeq 0,\, i =1,\ldots, m, \\
& L_{x_{\Dt_i}}^{\Dt_i,k}[y_{\Dt_i}]\succeq 0,\, i =1,\ldots, m,\\
& y\in\re^{\mathbb{U}_k},\, y_{0} = 1 . 
\end{array}
\right.
\ee
This is a semidefinite program.  
The dual problem of \eqref{eq:spcopmom} is the $k$th order sparse SOS relaxation
\be \label{eq:spcopsos}
\left\{
\begin{array}{cl}
\max & \quad \gamma\\
\st &  \quad  f -g(x;\lambda,\mu)-\gamma\in \qmod[x]_{spa,2k}, \\
&  \quad \lambda\in\re^{m_1},\, \mu\in \re_+^{m_2},\,
\gamma\in\re .
\end{array}
\right.
\ee 
As $k \to \infty$, the sequence of the primal dual pair \eqref{eq:spcopmom}--\eqref{eq:spcopsos} is called the
sparse Moment-SOS hierarchy for solving \eqref{eq:spa_pop}.
The optimal values of \reff{eq:spcopmom} and \reff{eq:spcopsos}
are denoted as $f_{k}^{smo}$ and $f_k^{spa}$ respectively. 
Recall that $f_{min}$ denotes the optimal value of \eqref{eq:spa_pop}.
By the weak duality, it holds 
\be \label{bd:mom-sos}
f_k^{spa} \, \le \, f_k^{smo} \, \le \,  f_{min} .
\ee
If $f_k^{spa} = f_{min}$ for some integer $k$,
we say the $k$th order SOS relaxation \reff{eq:spcopsos} is tight.
Similarly, the $k$th order moment relaxation \eqref{eq:spcopmom} is said to be tight if $f_k^{smo} = f_{min}$. 
The sparse Moment-SOS hierarchy is said to be {\it tight} (or to have {\it finite convergence}) if $f_k^{smo} = f_k^{spa} = f_{min}$ for some relaxation order $k$.

In the following, we investigate conditions for the sparse Moment-SOS relaxations \eqref{eq:spcopmom}--\eqref{eq:spcopsos} to be tight.
Recently, sparse Moment-SOS relaxations have been extensively studied in
\cite{HuangKang25,korda2025convergence,Lasspacov06,MagronWang23,vmjw21,NieDemmel09,QuTang24,CSTSSOS,waki2006sums}.
In particular, necessary and sufficient conditions
for tightness of sparse Moment-SOS relaxations are given in \cite{NieQuTangZhang25,NQTZ25mat}.
We remark that linear constraints in the sparse copositive optimization
problem \eqref{eq:spa_pop} can be dense, so it is not a traditional sparse polynomial optimization problem. In \eqref{eq:spcopsos}, the sparse quadratic module is used. The tightness of sparse SOS relaxations is characterized as follows.

\begin{thm} \label{thm:tight:char}
For the $k$th order sparse SOS relaxation \eqref{eq:spcopsos}, we have:
\begin{enumerate}
\item[(i)] For a relaxation order $k\ge k_0$, there exist
$\lambda\in\re^{m_1}$ and $\mu\in\re_+^{m_2}$ such that
\[
f - g(x; \lambda, \mu)-f_{min}\in \qmod[x]_{spa, 2k}
\]
if and only if there exist sparse polynomials $p_i\in\re[x_{\Dt_i}]_{2k}$
such that
\be   \label{eq:pfdecomp}
\left\{
\begin{array}{l}
	p_1+\cdots+p_m + g(x;\lambda,\mu)+f_{min} = 0,\\
	f_i+p_i\in\qmod_{\Dt_i}[x_{\Dt_i}]_{2k},\quad  i=1,\ldots, m.
\end{array}
\right.
\ee

\item[(ii)] Suppose \eqref{eq:pfdecomp} holds for some $\lambda\in\re^{m_1}$ 
and $\mu\in\re_+^{m_2}$. Then, the optimal value of \eqref{eq:spa_pop} is attainable
if and only if there exists $u\in K$ such that
$g(u;\lambda,\mu) = 0$ and
$f_i(u_{\Dt_i})+p_i(u_{\Dt_i}) = 0$ for all $i\in[m]$.

\item[(iii)] The $k$th order sparse SOS relaxation \eqref{eq:spcopsos} 
is tight if and only if for every $\epsilon>0$,
there exist $p_i\in\re[x_{\Dt_i}]_{2k}$,
$\lambda\in\re^{m_1}$ and $\mu\in\re_+^{m_2}$ such that
\be \label{eq:tight_not_attain}
\left\{\begin{array}{l}
	p_1+\cdots+p_m+g(x;\lambda,\mu)+f_{min} = 0,\\
	f_i+p_i+\epsilon\in \qmod_{\Dt_i}[x_{\Dt_i}]_{2k},
	\quad i = 1, \ldots, m.
\end{array}
\right.
\ee
\end{enumerate}
\end{thm}

We remark that the item (i) concerns tightness of  the sparse SOS relaxation \eqref{eq:spcopsos} when it achieves the optimal value,
while the item (iii) concerns tightness  when \eqref{eq:spcopsos} does not achieve the optimal value.

\medskip 
\noindent
{\it Proof of Theorem~\ref{thm:tight:char}}
(i) If \eqref{eq:pfdecomp} holds, then
\[
f - g(x;\lambda, \mu) - f_{min} = \sum_{i=1}^m (f_i+ p_i )
\in \qmod[x]_{spa,2k}.
\]
Conversely, suppose there exist polynomials $q_i\in \qmod_{\Dt_i}[x_{\Dt_i}]_{2k}$ such that
\[ f-g(x;\lambda,\mu)-f_{min} = q_1+\cdots+q_m .\]
Then \eqref{eq:pfdecomp} holds for $p_i = q_i-f_i$.

(ii) Suppose \eqref{eq:pfdecomp} holds.
If $u\in K$ is a common zero of $g(x;\lambda,\mu) = 0$ and
$f_i+p_i = 0$ for all $i\in[m]$,
then it is an optimizer of \eqref{eq:spa_pop} since
\[\begin{aligned}
	f(u) &= \sum\limits_{i=1}^m \big[f_i(u_{\Dt_i})+p_i(u_{\Dt_i}) \big]
	-\sum\limits_{i=1}^m p_i(u_{\Dt_i}) \\
	& = f_{min}+g(u;\lambda,\mu) = f_{min}.
\end{aligned}\]
For the ``only if'' direction, suppose $u\in K$ is an optimizer of \eqref{eq:spa_pop},
i.e., $f(u) = f_{min}$.
Then, by evaluating the equality in \eqref{eq:pfdecomp}
at $ x = u $, we get
\[
-g(u;\lambda,\mu) = f(u) + \sum\limits_{i=1}^m p_i(u_{\Dt_i}) =
\sum\limits_{i=1}^m [ f_i(u_{\Dt_i})+p_i(u_{\Dt_i}) ] \ge 0,
\]
where the inequality is implied by $u\in K$
and $f_i+p_i\in \qmod_{\Dt_i}[x_{\Dt_i}]_{2k}$.
Since $g(u;\lambda, \mu)\ge 0$ by its definition,
the above inequality must be an equality,
hence $g(u;\lambda, \mu)=0$ and
$f_i(u_{\Dt_i})+ p_i(u_{\Dt_i})=0$ for every $i$.

(iii) If \eqref{eq:tight_not_attain} holds, then 
\begin{eqnarray*}
 &&	f-g(x;\lambda,\mu)-(f_{min}-m\epsilon)  \\ 
&=&	f+\sum\limits_{i=1}^m p_i +m\epsilon
	= \sum\limits_{i=1}^m(f_i+p_i+\epsilon)
	\in\qmod[x]_{spa,2k}.
\end{eqnarray*} 
This implies $\gm = f_{min}-m\epsilon$ is feasible for \eqref{eq:spcopsos}.
Thus, for arbitrary $\epsilon > 0$, it holds 
\[ f_{min}\ge f_k^{spa}\ge f_{min}-m\epsilon .\]
So, the $k$th order sparse SOS relaxation \eqref{eq:spcopsos} is tight.
Conversely, if \eqref{eq:spcopsos} is tight, 
then $\gm = f_{min}- m\epsilon$ is feasible for \eqref{eq:spcopsos} for every $\epsilon>0$. 
Like in (i), we can similarly show that there exist polynomials $q_i\in\re[x_{\Dt_i}]_{2k}$ such that
\[
\left\{
\begin{array}{l}
	q_1 +\cdots+q_m +g(x;\lambda,\mu)+(f_{min}-m\epsilon) = 0,\\
	f_i+q_i\in \qmod_{\Dt_i}[x_{\Dt_i}]_{2k},\quad i = 1, \ldots, m.
\end{array}
\right.
\]
For each $i \in [m]$, let 
\[
p_i(x_{\Dt_i}) = q_i(x_{\Dt_i}) - \epsilon\in\re[x_{\Dt_i}]_{2k} .
\]
Then, \eqref{eq:tight_not_attain} follows from the above.
\qed
%%\end{proof}

In computational practice, the optimal value $f_{min}$ is usually not known in advance. So, we typically verify tightness by examining the optimizer of the moment relaxation \eqref{eq:spcopmom}.

\begin{thm}\label{thm:rank1_tight}
Suppose $y^*$ is a minimizer of the sparse moment relaxation
\eqref{eq:spcopmom} with order $k$.
If $\rank\, M_{\Dt_i}^{(k_0)}[y^*_{\Dt_i}] = 1$ for all $i\in [m]$,
then $f_k^{smo} = f_{min}$ and $x^* \coloneqq \pi(y^*)$ is a minimizer of \eqref{eq:spa_pop}.
\end{thm}
\begin{proof}
Note $x^*$ is a feasible point of \eqref{eq:spa_pop}.
If $\rank\, M_{\Dt_i}^{(k_0)}[y_{\Dt_i}^*] = 1$ for every $i$, then 
\[
y_{\Dt_i}^*|_{2k_0} = [x_{\Dt_i}^*]_{2k_0}, \quad 
\langle f_i, y_{\Dt_i}^*\rangle  \, =  \,  f_i(x^*_{\Dt_i}), 
\]
where $y_{\Dt_i}^*|_{2k_0}$ denotes the $(2k_0)$th degree truncation of $y_{\Dt_i}^*$ 
as in \reff{trun:2k0}. Since $2k\ge 2k_0\ge \deg(f)$, it holds
\[
f_k^{smo}
= \sum\limits_{i=1}^m \langle f_i, y_{\Dt_i}^*\rangle
= \sum\limits_{i=1}^m f_i(x^*_{\Dt_i}) = f(x^*)\ge f_{min}.
\]
Since $f_k^{smo}\le f_{min}$ by \reff{bd:mom-sos},
the above implies $f_k^{smo} = f_{min}$ and this optimal value is
achieved at $x^*$. So, $x^*$ is a global minimizer.
\end{proof}

The following is an example to illustrate the above theorem.

\begin{Example} \rm 
Consider the sparse copositive optimization problem
\[
\left\{
\baray{cl}
\min & f(x) = f_1(x_{\Dt_1}) + f_2(x_{\Dt_2}) + f_3(x_{\Dt_3}) \\
\st &   x_1 + x_2 + x_3 = 3 , \\    
	&   2x_1 + x_2 + x_3 = 4,\, x \in \re_+^3,
\earay
\right.
\]
where $\Dt_1 = \{1,2\}$, $\Dt_2 = \{2,3\}$,  $\Dt_3 = \{1,3\}$
and
\[
\begin{aligned}
	f_1(x_1, x_2) &= x_1^4 + x_1^2x_2^2 - 2x_1^2x_2 - 2x_1x_2 + x_2^2 + x_1 , \\
	f_2(x_2, x_3) &= x_2^4 + x_2^2x_3^2 - 2x_2^2x_3 - 2x_2x_3 + x_3^2 + 3x_2 , \\
	f_3(x_1, x_3) &= x_3^4 + x_1^2x_3^2 - 2x_1x_3^2 - 2x_1x_3 + x_1^2 + 3x_3 .
\end{aligned}
\]
Note $\deg(f) = 4$, $k_0 = 2$ and there is no Lagrange multiplier $\mu$ for $Cx \ge d$. One can verify that $\gamma = 4$ and $\lambda = (5,-2)$ are feasible
for the SOS relaxation \eqref{eq:spcopsos} for $k=2$, since   
\begin{eqnarray*}
& & f(x)-g(x;\lambda)-\gamma \\
&=& f(x) -(5(x_1+x_2+x_3-3) - 2(2x_1+x_2+x_3-4)) - 4   \\
&=& \underbrace{(x_1^2 - x_2)^2 + (x_1x_2 - 1)^2}_{\in \qmod_{\Dt_1}[x_{\Dt_1}]_{4}}
+ \underbrace{(x_2^2 - x_3)^2 + (x_2x_3 - 1)^2}_{\in \qmod_{\Dt_2}[x_{\Dt_2}]_{4}}\\
&& \qquad + \underbrace{(x_3^2 - x_1)^2 + (x_1x_3 - 1)^2}_{\in\qmod_{\Dt_3}[x_{\Dt_3}]_{4}}.
\end{eqnarray*}
Hence, $\gamma = 4 \le f_{min}$.
Solving the sparse moment relaxation \eqref{eq:spcopmom} for $k=2$,
we get $f_2^{smo} = 4$ and the corresponding
optimizer $y^*\in\re^{\mathbb{U}_2}$ satisfies
\[
\begin{gathered}[c]  
	\rank\,M_{\Dt_1}^{(2)}[y^*_{\Dt_1}] \,=\,
	\rank\, M_{\Dt_2}^{(2)}[y^*_{\Dt_2}]
	\,=\, \rank\,  M_{\Dt_3}^{(2)}[y^*_{\Dt_3}]= 1,  \\ 
	\pi(y^*) = (y_{100},\, y_{010},\, y_{001}) =  (1, 1, 1). 
\end{gathered}
\]
By Theorem~\ref{thm:rank1_tight}, we know $f_{min} = 4$
and $\pi(y^*) = (1,1,1)$ is the global optimizer for this example.
\end{Example}

The rank one condition in Theorem~\ref{thm:rank1_tight} can be weakened, to verify tightness of the sparse moment relaxation \eqref{eq:spcopmom}.
When $M_{\Dt_i}^{(k_0)}[y_{\Dt_i}^*]$ has rank higher than one, 
the truncation $y_{\Dt_i}^*|_{2k_0}$ may still admits a finitely atomic measure supported in $\re_+^{\Dt_i}$. 
In computational practice, we can use the {\it flat truncation} \cite{NieFlat13} to verify this. 
Suppose there is an integer $t\in[k_0,k]$ such that
\be \label{FT:t-1}
\rank \, M_{\Dt_i}^{(t-1)}[y_{\Dt_i}^*] \,\,  = \,\,  \rank \, M_{\Dt_i}^{(t)}[y_{\Dt_i}^*] = r_i > 0, 
\ee
then there exist $r_i$ distinct points $u^{(i,j)} \in K$ and scalars $\theta_{i,j}>0$ such that
\be   \label{eq:y_delta_decomp}
y_{\Dt_i}^*|_{2t} = \sum_{j=1}^{r_i} \theta_{i,j} [u^{(i,j)}]_{2t},\quad
\sum\limits_{j=1}^{r_i} \theta_{i,j} = 1.
\ee 
We refer to \cite{NieFlat13} for the above. 
When \eqref{eq:y_delta_decomp} holds, we denote the support
\[
\supp{y_{\Dt_i}^*|_{2t}} \, \coloneqq \,  \{u^{(i,1)},\ldots, u^{(i,r_i)}\}.
\]
It is uniquely determined by the flat truncation condition 
(see \cite[Theorem~2.7.1]{NieBook}). The decomposition \eqref{eq:y_delta_decomp} alone is not enough to certify the tightness of the sparse moment relaxation \eqref{eq:spcopmom}. We need some compatibility assumption about points in $\supp{y_{\Dt_i}^*|_{2t}}$ to ensure its tightness.

\begin{thm}  \label{thm:multipoints}
Assume \eqref{eq:spa_pop} has no inequality constraint $Cx\ge d$ but $x \ge 0$ still appears. Suppose $y^*$is a minimizer of the sparse moment relaxation \eqref{eq:spcopmom} 
and there exists $t\in [k_0,k]$ such that \eqref{eq:y_delta_decomp} holds for all $i\in[m]$. If there is a point $x^*\in K$ such that $x_{\Dt_i}^*\in \supp{ y_{\Dt_i}^*|_{2t} }$ for all $i$ and $f_t^{smo} = f_k^{smo}$, then $x^*$ is a minimizer of \eqref{eq:spa_pop}.
\end{thm}
\begin{proof}
Since \eqref{eq:y_delta_decomp} holds for each $i$, we have
\[
M_{\Dt_i}^{(t)}[y^*_{\Dt_i}] = \rho_i [x_{\Dt_i}^*]_{t}[x_{\Dt_i}^*]_{t}^T + W_{\Dt_i}
\]
for a positive scalar $\rho_i >0$ and a sparse moment matrix $W_{\Dt_i} \succeq 0$. Let
\[
\rho \coloneqq \min_{1 \le i \le m} \rho_i.
\]
Note $0 < \rho \le 1$. If $\rho = 1$, then $W_{\Dt_i} = 0$ and $\rank \, M_{\Dt_i}^{(t)}[y^*_{\Dt_i}]  = 1$ for every $i$, 
and the conclusion holds by Theorem~\ref{thm:rank1_tight}. 
In the following, we consider the case $0<\rho<1$.
Define 
\[
\hat{y} \coloneqq (\hat{y}_\alpha)_{\alpha \in \mathbb{U}_t} \quad \text{where each} \quad
\hat y_{\alpha}= \left(x_{\Dt_i}^*\right)^{\alpha} .
\]
Let $\tilde y \coloneqq  (\tilde y_\alpha)_{\alpha\in \mathbb{U}_t}$ be the tms such that
\[ y^*|_{2t} = \rho \hat y +(1-\rho)\tilde y . \]
Since $x^*\in K$, it is clear that $\hat{y}$ is feasible for (\ref{eq:spcopmom}) with the relaxation order equal to $t$.
As in the proof of Theorem~3.3 of \cite{NieQuTangZhang25}, we can show
\[
M_{\Dt_i}^{(t)}[\tilde{y}_{\Dt_i}] \succeq 0,\quad 
L_{x_{\Dt_i}}^{\Dt_i, t}[\tilde{y}_{\Dt_i}]\succeq 0,\quad i =1,\ldots, m. 
\]
Since $\tilde y = \frac{1}{1-\rho}(y^* - \rho \hat y)$, it holds
\[  
A\pi(\tilde{y}) = \frac{1}{1-\rho}\big( A\pi(y^*)-\rho A\pi(\hat y) \big) 
= \frac{1}{1-\rho}(b-\rho b) = b. 
\]
Thus, $\tilde{y}$ is also feasible for (\ref{eq:spcopmom}) with the relaxation order equal to $t$.
Since $f_t^{smo} =f_k^{smo} = \langle f, y^* \rangle$ and $t\ge k_0$, we have
\[
\langle f, y^* \rangle \le \langle f, \hat{y} \rangle,  \quad
\langle f, y^* \rangle \le \langle f, \tilde{y} \rangle.
\]
Since
$\langle f, y^* \rangle = \rho \langle f, \hat{y} \rangle
+ (1 - \rho) \langle f, \tilde{y} \rangle$ and $0<\rho<1$, it holds
\[
f_k^{smo} = \langle f, y^* \rangle = \langle f, \hat{y} \rangle = \langle f, \tilde{y} \rangle.
\]
Note that $\langle f, \hat{y} \rangle = f(x^*)$.
The conclusion then follows from (\ref{bd:mom-sos}).
\end{proof}

\begin{rmk}
In Theorem~\ref{thm:multipoints}, the conclusion may not hold if problem (\ref{eq:spa_pop}) has the additional inequality constraint $Cx \ge d$. When (\ref{eq:spa_pop}) has such inequality constraints, the relaxations \eqref{eq:spcopmom}--\eqref{eq:spcopsos} might not be tight and the point $x^*\in K$ in Theorem~\ref{thm:multipoints} may not be a minimizer, even if all other conditions in the theorem are satisfied.
For example, consider the case that $m = n = 1$ and
\[ f(x) = x^4-3x^2,\quad K = \{x\in \re_+^1: x\le 0.5\}.\]
The minimum value $f_{min} = -\frac{11}{16}$.
When $k = 2$, for $\mu = 2$, $\gm = -1$, we have
\[ f(x) - \mu (0.5-x) - \gm = (x^2-x)^2 + 2x(x-1)^2 \in \qmod[x]_4.\]
So, $f_2^{smo} \ge f_2^{spa} \ge -1$. 
On the other hand, if we let 
\[ y^* = \frac{1}{2}\big( [0]_4 + [1]_4 \big), \] 
then $y^*$ is feasible for (\ref{eq:spcopmom}) 
and $\langle f, y^* \rangle = -1$.
Thus, $y^*$ is a minimizer of the sparse moment relaxation (\ref{eq:spcopmom}) with $k = 2$, and (\ref{eq:y_delta_decomp}) holds for $t = k = 2$.
Also, $x^* = 0$ is in the support of $y^*$ and it is feasible for (\ref{eq:spa_pop}).
However, the minimum value of (\ref{eq:spcopmom}) is $-1 < -\frac{11}{16}$
and $f(x^*) = 0$, so the relaxations \eqref{eq:spcopmom}--\eqref{eq:spcopsos} are not tight and $x^*$ is not a minimizer.
\end{rmk}

In the following, we give a trick for checking tightness of the relaxations \eqref{eq:spcopmom}--\eqref{eq:spcopsos} and extracting minimizers for \reff{eq:spa_pop}, regardless of the conditions in Theorem~\ref{thm:multipoints} hold or not. Let $y^*$ be the minimizer of \eqref{eq:spcopmom} and suppose there exists $t\in[k_0,k]$ such that (\ref{FT:t-1}) holds for every $i=1,\ldots,m $. We apply the method in \cite{henrion2005detecting} to get the decomposition (\ref{eq:y_delta_decomp}) fo each $y_{\Dt_i}^*$.
Note that the decomposition is unique and $\supp{y_{\Dt_i}^*|_{2t}}$ 
of each $y_{\Dt_i}^*$ is a finite set.
There are finitely many points $x^*\in K$ such that every $x_{\Dt_i}^*\in \supp{ y_{\Dt_i}^*|_{2t} }$.
For every such $x^*\in K$, we check if $f(x^*) = f^{smo}_k$ holds or not.
If it holds for some $x^*$, then $f(x^*) = f_{min} = f^{smo}_k$, 
i.e., the relaxation \eqref{eq:spcopmom} is tight and $x^*$ is a global minimizer for \reff{eq:spa_pop}. The following example is an exposition for this trick.

\begin{Example}\rm
\label{ex:theorem_illustration_n4}  \rm 
Consider the sparse copositive optimization problem:
\begin{equation}
\left\{
\begin{array}{cl}
	\min  & f(x) =  f_1(x_{\Dt_1}) + f_2(x_{\Dt_2}) + f_3(x_{\Dt_3}) \\
	\st & x_1 + 2x_2 + 2x_3 + x_4 = 9, \\
	& x_1 + x_4 \ge 3, \, x\in\re_+^4.
\end{array}
\right.
\end{equation}
Here, $\Dt_1=\{1,2\}$, $\Dt_2=\{2,3\}$, $\Dt_3=\{3,4\}$, and
\begin{align*}
	f_1(x_1,x_2) &= x_1^2 x_2^2 + x_1^2 - 2x_1 x_2 + x_2^2 - 6x_1 - 6x_2, \\
	f_2(x_2,x_3) &= x_2^2 x_3^2 + x_2^2 - 2x_2 x_3 + x_3^2 - 6x_2 - 6x_3, \\
	f_3(x_3,x_4) &= x_3^2 x_4^2 + x_3^2 - 2x_3 x_4 + x_4^2 - 6x_3 - 6x_4.
\end{align*}
For the order $k= 2$, we solve the sparse moment relaxation \reff{eq:spcopmom} and get $f_2^{smo} = -39$.
The optimizer $y^*$ satisfies \reff{FT:t-1} and \eqref{eq:y_delta_decomp} with
\begin{align*}
y_{\Dt_1}^* &= \frac{1}{2}  \begin{bmatrix} 1\\2 \end{bmatrix} _{4} + \frac{1}{2}  \begin{bmatrix} 2\\1 \end{bmatrix} _{4}, \quad 	
\supp{y_{\Dt_1}^*} = \left\{ \begin{bmatrix} 1 \\ 2 \end{bmatrix}, \begin{bmatrix} 2 \\ 1 \end{bmatrix} \right\},\\
y_{\Dt_2}^* &= \frac{1}{2}  \begin{bmatrix} 2\\1 \end{bmatrix} _{4} + \frac{1}{2} \begin{bmatrix} 1\\2 \end{bmatrix} _{4}, \quad
\supp{y_{\Dt_2}^*} = \left\{ \begin{bmatrix} 2 \\ 1 \end{bmatrix}, \begin{bmatrix} 1 \\ 2 \end{bmatrix} \right\},\\
y_{\Dt_3}^* &= \frac{1}{2}  \begin{bmatrix} 1\\2 \end{bmatrix} _{4} + \frac{1}{2}  \begin{bmatrix} 2\\1 \end{bmatrix}_{4},\quad
\supp{y_{\Dt_3}^*}  = \left\{ \begin{bmatrix} 1 \\ 2 \end{bmatrix}, \begin{bmatrix} 2 \\ 1 \end{bmatrix} \right\}.
\end{align*}
From these support sets, we can extract two points
\[
x^{*,1} = (1,2,1,2),\quad x^{*,2} = (2,1,2,1).
\] 
They are both feasible points for this sparse copositive optimization problem and
$x_{\Dt_i}^{*,1}, x_{\Dt_i}^{*,2}\in \supp{y_{\Dt_i}^*}$ for all $i$.
Moreover, it holds
\[ f(x^{*,1}) = f(x^{*,2}) = -39 = f_2^{smo}. \]
Therefore, the relaxation \reff{eq:spcopmom} is tight, and both
$x^{*,1}$ and $x^{*,2}$ are global minimizers.
\end{Example}

When the assumptions of Theorem~\ref{thm:multipoints} are not satisfied,
the point $\pi(y^*)$ can still serve as a candidate optimizer.
An interesting question is what are sufficient conditions to ensure that  
$\pi(y^*)$ is a global optimizer of \reff{eq:spa_pop}. This is explored in the following section.

\section{The cop-SOS convexity}
\label{sec:tightness}

In this section, we show that the sparse Moment-SOS relaxations 
\eqref{eq:spcopmom}--\eqref{eq:spcopsos} are tight under the assumption of cop-SOS convexity.

We consider the case that every $f_i(x_{\Dt_i})$ is a convex polynomial in $x_{\Dt_i}$. Then \eqref{eq:spa_pop} becomes a linearly constrained convex optimization problem.
If it is bounded below (i.e., $f_{min} > - \infty$),
then \eqref{eq:spa_pop} must achieve the minimum value $f_{min}$
and has a global minimizer, say, $u$. We refer to \cite{BelKla02} for this.
Since all constraints in \eqref{eq:spa_pop} are linear,
the Karush-Kuhn-Tucker (KKT) conditions must hold \cite[Proposition~3.3.1]{DPB}: 
there exist Lagrange multiplier vectors $\lambda\in\re^{m_1}$, $\mu\in\re_+^{m_2}$ and $\eta\in\re_+^n$ such that
\be   \label{eq:KKTcondi}
\boxed{
\begin{array}{c}
	\nabla f(u) = \sum\limits_{i=1}^m \nabla f_i(u_{\Dt_i}) = A^T\lambda+C^T\mu+\eta,\\
	0\le \mu\perp (Cu-d)\ge 0,\, 0\le \eta \perp u \ge 0.
\end{array}
}
\ee
In the above, $\perp$ means the vectors are perpendicular.

\begin{thm} \label{thm:gen-cvx}
Suppose every $f_i(x_{\Dt_i})$ is convex and $u$ is an optimizer of \eqref{eq:spa_pop}. Then, there exist polynomials $p_i\in\re[x_{\Dt_i}]$ and there exists a Lagrange multiplier vector triple
$(\lambda,\mu,\eta)\in\re^{m_1}\times \re_+^{m_2}\times \re_+^n$
such that \reff{eq:KKTcondi} holds and 
\be  \label{eq:tightness_assump}
\boxed{
\begin{array}{c}
 p_1+\cdots+p_m +g(x;\lambda, \mu) + f_{min} = 0, \\
 f_i+p_i\ge 0\,\,\mbox{on $\re_+^{\Dt_i}$},\, i=1,\ldots,m.
\end{array}
}
\ee
\end{thm}
\begin{proof}
Since all constraints are linear, the KKT conditions \eqref{eq:KKTcondi} must hold,  i.e., there exists  $(\lambda,\mu,\eta)$ satisfying \eqref{eq:KKTcondi}. 
We can write that 
\[ \eta = \eta_1+\cdots+\eta_m, \quad \text{where each} \quad 
\eta_i \in \re_+^{\Dt_i} .
\]
Here, each $\re_+^{\Dt_i}$ denotes the set of nonnegative vectors in $\re^n$ 
whose indices of positive entries are contained in $\Dt_i$. For each $i=1,\ldots, m$, let
\[
p_i(x) \, \coloneqq \, -(x-u)^T\big( \nabla f_i(u_{\Dt_i})-\eta_i \big)-f_i(u_{\Dt_i}).
\]
Since $f_i$ only depends on $x_{\Dt_i}$, $p_i$ is a polynomial only in $x_{\Dt_i}$. The complementarity   $0\le u\perp \eta \ge 0$ 
implies $u^T\eta_i = 0$ for all $i\in[m]$.
Since $f_i(x_{\Dt_i})$ is convex in $x_{\Dt_i}$, it holds
\[
f_i(x_{\Dt_i})-f_i(u_{\Dt_i})-(x-u)^T\nabla f_i(u_{\Dt_i}) \ge 0.
\]
Hence, we have
\[
\begin{aligned}
f_i(x_{\Dt_i})+p_i(x_{\Dt_i}) &= f_i(x_{\Dt_i})-f_i(u_{\Dt_i})-(x-u)^T (\nabla f_i(u_{\Dt_i})-\eta_i) \\
&= [f_i(x_{\Dt_i})-f_i(u_{\Dt_i})-(x-u)^T\nabla f_i(u_{\Dt_i}) ] + (x-u)^T\eta_i \\
& \ge (x-u)^T\eta_i  = x^T\eta_i \ge 0 
\end{aligned}
\]
for all $x\in\re_+^n$. This shows the second line in \eqref{eq:tightness_assump}.
Since $Au = b$, the KKT conditions in \reff{eq:KKTcondi} imply
\[
g(u;\lambda,\mu) = (Au-b)^T\lambda + (Cu-d)^T\mu = 0, 
\quad u^T \eta  = 0.
\] 
Since $\sum_{i=1}^m f_i(u_{\Dt_i}) = f_{min}$, it holds
\[
\sum\limits_{i=1}^m p_i(x) = -(x-u)^T(A^T\lambda+C^T\mu) -f_{min}.
\]
Then, one can verify that
\[
\sum\limits_{i=1}^m p_i(x) = g(u;\lambda, \mu)-g(x;\lambda,\mu)-f_{min}
= -g(x;\lambda,\mu)-f_{min}.
\]
Therefore, \eqref{eq:tightness_assump} holds.
\end{proof}

When conditions \eqref{eq:tightness_assump} and \eqref{eq:pfdecomp} are compared, there is a difference between the nonnegativity $f_i+p_i\ge 0$ on $\re_+^{\Dt_i}$ and the quadratic module membership $f_i+p_i\in \qmod_{\Dt_i}[x_{\Dt_i}]_{2k}$. We need stronger conditions to ensure the membership; see as in \cite{DeKlerk11} and \cite[Chap. 7]{NieBook}.

In polynomial optimization, the SOS-convexity is a useful property to prove tightness of Moment-SOS relaxations \cite[Chap.~7]{NieBook}.
A polynomial $p \in\re[x]$ is said to be {\it SOS convex} if its Hessian
$\nabla^2 p(x)$ is a sum of Hermitian squares, i.e., 
there exists a matrix polynomial $P(x)$ such that
\[
\nabla^2 p(x) \,\,  = \,\,  P(x) P(x)^T.
\]
The SOS convexity can be generalized to copositive optimization.

\begin{defi} \label{def:cop-SOS} 
A polynomial $p \in \re[x]$ is said to be cop-SOS convex if
there exist matrix polynomials $P_0(x)$, $P_1(x)$, \ldots, $P_n(x)$ such that
\be   \label{eq:cop-SOS}
\nabla^2 p(x) \, =  \, P_0(x)P_0(x)^T + x_1 P_1(x)P_1(x)^T + \cdots + x_n P_n(x)P_n(x)^T.
\ee
Similarly, $p$ is said to be cop-SOS concave if $-p$ is cop-SOS convex.
\end{defi}

Note that $p$ is cop-SOS convex if and only if its Hessian $\nabla^2 p(x)$ 
belongs to the matrix quadratic module of $(x_1,\ldots, x_n)$.
Hence, the cop-SOS convexity can be verified by solving a semidefinite program.
We refer to \cite[Chap.~10]{NieBook} for how to do this. 
A very interesting case is that $p$ is a cubic polynomial, and its Hessian can be written as 
\be   \label{cubic:cop-SOS}
\nabla^2 p(x) = A_0 + x_1 A_1 + \cdots + x_n A_n,
\ee
for some symmetric matrices $A_i$. 
Note that $p(x)$ is convex in $\re_+^n$ if and only if all $A_i \succeq 0$ are positive semidefinite. 
Every $A_i \succeq 0$ can be written as $P_iP_i^T$ for some matrix $P_i$. 
Therefore, a cubic polynomial $p$ is convex in $\re_+^n$ if and only if it is cop-SOS convex.

A useful property of SOS-convexity is that a type of Jensen's inequality holds 
(see \cite{Las09} or \cite[Theorem~7.1.6]{NieBook}). 
It also holds for cop-SOS convex polynomials.

\begin{lem} \label{lm:jensen_copositive}
(i) Let $\Dt = \{1,\ldots, n\}$ and let $y = (y_{\alpha})\in\re^{\N_{2k}^{\Dt}}$ be a tms such that 
\[
y_0=1, \quad  M_{\Dt}^{(k)}[y_{\Dt}]\succeq 0, \quad  L_{x_{\Dt}}^{\Dt, k}[y_{\Dt}] \succeq 0 .
\]
If $p \in \re[x_{\Dt}]_{2k}$ is cop-SOS convex, 
then for $u = \pi(y)$ it holds
\be \label{Jesen:cop-sos}
p( u )  \le  \langle p, y_{\Dt}\rangle.
\ee
(ii) For $\Dt_i = \{j_1,\ldots, j_{n_i}\}\subseteq [n]$,
let $y = (y_{\alpha})\in\re^{\N_{2k}^{\Dt_i}}$ be a tms such that
\[
y_0=1, \quad M_{\Dt_i}^{(k)}[y_{\Dt_i}]\succeq 0, \quad L_{x_{\Dt_i}}^{\Dt_i,k}[y_{\Dt_i}] \succeq 0 .
\]
If $f_i\in\re[x_{\Dt_i}]_{2k}$ is cop-SOS convex,
then for $u = \pi(y)$ it holds
\be  \label{Jesen:Dti}
f_i(u_{\Dt_i}) \, \le \, \langle f_i, y_{\Dt_i} \rangle.
\ee
\end{lem}
\begin{proof}
(i) By the first order Taylor expansion formula, we have
\[
p(x) = p(u) + \nabla p(u)^T (x-u) +  Q(x,u)  {\color{red},}
\]
where  
\[
Q(x,u)  \, = \, \int_0^1 (1-t) (x-u)^T
\big[ \nabla^2 p( u + t(x-u) ) \big] (x-u) \mathtt{d}t .
\]
Since $p$ is cop-SOS convex, say, \reff{eq:cop-SOS} holds, we have
\begin{align*}
	Q(x, u)  = & \int_0^1 (1-t) (x-u)^T
	\Big[ P_0(  tx +(1-t)u ) P_0( tx +(1-t)u )^T + \\
	& \sum_{i=1}^n (tx_i +(1-t)u_i) 
	P_i(  tx +(1-t)u ) P_i(  tx +(1-t)u )^T 
	\Big] (x-u) \mathtt{d}t .
\end{align*}
This implies 
$
Q(x, u) \in \qmod_{\Dt}[x_{\Dt}]_{2k},
$
so 
\[
\langle p  - p(u) - \nabla p(u)^T(x-u),
y \rangle = \langle Q(x,u), y_{\Dt} \rangle \geq 0 .
\]
The above inequality is due to \reff{dual:QM:Dti}.
Therefore, we get
\[
\langle p, y_{\Dt}  \rangle \geq
\langle p(u) + \nabla p(u)^T(x-u),  y_{\Dt} \rangle
= p(u).
\]
In the above, the equality follows from $u = \pi(y)$ and
\[
\begin{gathered}
	\langle p(u),  y_{\Dt} \rangle =
	p(u) \langle 1,  y_{\Dt} \rangle  = p(u)  y_0 = p(u), \\
	\langle \nabla p(u)^T(x-u),  y_{\Dt} \rangle
	=  \nabla p(u)^T(u-u) = 0.	
\end{gathered}
\]
So, the inequality \reff{Jesen:cop-sos} holds.

(ii) This follows from (i) if we view $\Dt_i$ as $\Dt$.
\end{proof}

Under the cop-SOS convexity, we show the sparse moment relaxation \eqref{eq:spcopmom} is tight for all orders $k\ge k_0$. Recall that
$f_{min}$ denotes the optimal value of \eqref{eq:spa_pop}
and $f_k^{smo}$ is the optimal value of \eqref{eq:spcopmom} for the order $k$.
A point $x$ is said to be {\it strictly} feasible for \eqref{eq:spa_pop} if 
$Ax=b$, $Cx > d$ and $x>0$.

\begin{thm}\label{thm:cop-sos-cvx}
Assume $f_i\in\re[x_{\Dt_i}]$ is cop-SOS convex for every $i\in[m]$. 
Suppose $(\gamma^*, \lambda^*,\mu^*)$ is an optimizer of \eqref{eq:spcopsos} and $y^*$ 
is an optimizer of \eqref{eq:spcopmom} at a relaxation order $k\ge k_0$. 
Then, $f_{k}^{smo} = f_{min}$ and the point $\pi(y*) = (y_{e_1}^*,\ldots, y_{e_n}^*)$ is a minimizer of \eqref{eq:spa_pop}.
Moreover, if, in addition, \eqref{eq:spa_pop} has a strictly feasible point,
then $\gamma^* = f_{min}$ and
\[
f-g(x;\lambda^*, \mu^*)-f_{min} \, \in \, \qmod[x]_{spa,2k_0} .
\] 
\end{thm}
\begin{proof}
Let $x^* = \pi(y^*)$. Since each $f_i$ is cop-SOS convex, Lemma~\ref{lm:jensen_copositive} implies
\[ 
f_i(x_{\Dt_i}^*)  \le \langle f_i,y_{\Dt_i}^*\rangle .
\] 
Since $x^*$ is feasible for \eqref{eq:spa_pop}, we get
\[
f_{min}\le f(x^*) = \sum\limits_{i=1}^m f_i(x_{\Dt_i}^*)
\le \sum\limits_{i=1}^m \langle f_i, y_{\Dt_i}^*\rangle
= f_{k}^{smo}.
\]
On the other hand, $f_{k}^{smo}\le f_{min}$ by \reff{bd:mom-sos},
so $f_{min} = f_{k}^{smo}$ and the minimum value is achieved at
$x^*$. In addition, if \eqref{eq:spa_pop} has a strictly feasible point, then the sparse moment relaxation \eqref{eq:spcopmom} is also strictly feasible
(this can similarly shown as in \cite[Theorem~2.5.2]{NieBook}).
In this case, the strong duality holds between \eqref{eq:spcopmom} and \eqref{eq:spcopsos},
which implies $f_{k}^{spa} = f_{k}^{smo} = f_{min}$.
\end{proof}

\section{Numerical experiments}
\label{sec:num}

We provide numerical experiments for the sparse Moment-SOS relaxations (\ref{eq:spcopmom})-(\ref{eq:spcopsos}).
They are implemented in {\tt Yalmip} \cite{yalmip},
which calls the SDP package {\tt Mosek} \cite{mosek}
with the default settings of parameters.
The computations are implemented in MATLAB R2022b on a Lenovo
Laptop with CPU@2.10GHz and RAM 16.0G.
For neatness, only four decimal digits are displayed for computational results.

\begin{Example}\rm
Consider the sparse copositive optimization problem
\[
\left\{
\baray{cl}
\min &f(x) = 
\underbrace{x_1^4 + x_2^4 - 10 x_1^2 x_2^2}_{f_1(x_{\Dt_1})} + 
\underbrace{6 x_2^4 + x_3^4 + 2 x_2^2 x_3^2}_{f_2(x_{\Dt_2})} + 
\underbrace{6 x_1^4 + x_3^4 + 2 x_1^2 x_3^2}_{f_3(x_{\Dt_3})} \\
\st &  x_1 + x_2 + x_3 =1,\,  3x_1 - x_2 + 2x_3 \ge 0.5, \\ & x_1 + 2x_2 \le 1.2,\, x\in\re_+^3,
\earay\right.
\]
with $\Dt_1=\{1,2\}$, $\Dt_2=\{2,3\}$, and $\Dt_3=\{1,3\}$.
%
%The function $f_1$ is not convex in $\re_+^2$, but $f_2$ and $f_3$ are.
%
We solve the sparse Moment-SOS relaxations \reff{eq:spcopmom}
and \reff{eq:spcopsos} for the initial relaxation order $k=2$.
For the optimizer $y^*$, we have
\[ 
\rank\, M_{\Dt_1}^{(2)}[y_{\Dt_1}^*] = 
\rank\, M_{\Dt_2}^{(2)}[y_{\Dt_2}^*] = 
\rank\, M_{\Dt_3}^{(2)}[y_{\Dt_3}^*] = 1.
\]
By Theorem~\ref{thm:rank1_tight}, this sparse relaxation is tight and
\[
f_{min} = f^{spa}_2 \approx 0.1201 .
\]
The obtained global minimizer is 
\[
x^* = \pi(y^*) \,\approx\, (0.3567,\,  0.3567,\,  0.2867).
\]
It took around 0.61 seconds.
\end{Example}

\begin{Example}\rm
\label{ex:odd_degree_8}
Consider the sparse copositive optimization problem
\[
\left\{\begin{array}{rl}
\min & f(x) \coloneqq f_1(x_{\Dt_1}) + f_2(x_{\Dt_2})+f_3(x_{\Dt_3}) + f_4(x_{\Dt_4}) \\
\st & x_1 - x_2 + x_3 - x_4 + x_5 - x_6 + x_7 - x_8 \ge 0, \\
& \sum\limits_{i=1}^8 i x_i = 1, \, x\in\re_+^8,	
\end{array}\right.
\]
where $\Dt_1 = \{1, 2, 3\}$, $\Dt_2 = \{3, 4, 5\}$,
$\Dt_3 = \{5, 6, 7\}$, $\Dt_4 = \{7, 8\}$ and
\[
\begin{aligned}
f_1(x_1,x_2.x_3) &= (x_1 + x_2)^3 + x_1^5 + x_2^2 + (x_2 - x_3)^4, \\
f_2(x_3,x_4,x_5) &= (x_3 + x_4)^3 + x_3^5 + x_4^2 + (x_4 - x_5)^4, \\
f_3(x_5,x_6,x_7) &= (x_5 + x_6)^3 + x_5^5 + x_6^2 + (x_6 - x_7)^4, \\
f_4(x_7,x_8) &= (x_7 + x_8)^3 + x_7^5 + x_8^2 + x_8^3.
\end{aligned}
\]
In observation like \reff{cubic:cop-SOS}, one can check that every $f_i$ is cop-SOS convex. We do this for $f_1$. Its even degree terms $x_2^2$ and $(x_2 - x_3)^4$ are SOS convex. The Hessian of the 5th-degree term can be written as $\nabla^2(x_1^5) = x_1 \cdot \operatorname{diag}(20x_1^2, 0, 0)$. 
The Hessian of the 3rd-degree term is a sum of SOS matrices by $x_1$ and $x_2$
\[
\nabla^2(x_1+x_2)^3 = 6x_1 \begin{bmatrix} 1 & 1 & 0 \\ 1 & 1 & 0 \\ 0 & 0 & 0 \end{bmatrix} 
+ 6x_2 \begin{bmatrix} 1 & 1 & 0 \\ 1 & 1 & 0 \\ 0 & 0 & 0 \end{bmatrix}.
\]
Consequently, $\nabla^2 f_1(x_{\Dt_1})$ has the decomposition as in \eqref{eq:cop-SOS}.  The cop-SOS convexity of other $f_i(x_{\Dt_i})$ can be verified similarly.
Since $\deg(f) = 5$, the initial relaxation order
$k_0 = \lceil 5/2 \rceil = 3$.
By Theorem~\ref{thm:cop-sos-cvx}, the sparse Moment-SOS relaxations
\reff{eq:spcopmom} and \reff{eq:spcopsos} are tight for all $k \ge 3$.
Solving them for $k =3$, we get  
$
f_{min}=f^{spa}_3 \approx 0.0008,
$
and the obtained optimizer is 
\[
x^* \, \approx \, (0.0272,\,0.0014,\,0.0465,\,0.0018,\,
0.0596,\,0.0023,\,0.0708,\,0.0019).
\]
It took around $0.76$ seconds.
\end{Example}

\begin{Example}\rm \label{ex:odd_cycle_5}
Consider the sparse copositive optimization problem
\[
\left\{
\baray{cl}
\min  & x_1x_2 + x_2x_3 + x_3x_4 + x_4x_5 + x_5x_1 \\
\st & x_1 - x_2 + x_3 - x_4 + x_5 = 1, \\
& 2x_1 + x_2 - x_3 + 2x_4 - x_5 \ge 3, \\
&x_1+\cdots+x_5 =5,  \, x\in\re_+^5.
\earay
\right.
\]
The objective is given by $f_i = x_i x_{i+1}$ for $i=1,2,3,4$
and $f_5 = x_5x_1$. The sparsity label sets are 
\[
\Dt_1 = \{1,2\}, \quad
\Dt_2 = \{2,3\}, \quad
\Dt_3 = \{3,4\}, \quad
\Dt_4 = \{4,5\}, \quad
\Dt_5 = \{5,1\}.
\]
%%The sparse objective polynomials are not convex.
We solve the sparse Moment-SOS relaxations \reff{eq:spcopmom}-\reff{eq:spcopsos} 
for different relaxation orders $k$.
The computational results are shown in Table~\ref{ta:ex_52}.
\begin{table}[htbp]
\caption{Computational results for Example~\ref{ex:odd_cycle_5}.}
\label{ta:ex_52}
\vspace{.1 in} 
\centering
\begin{tabular}{lccccc}
	\specialrule{.2em}{0em}{0.1em}
	$k$ & 1 & 2 & 3 & 4 & 5  \\
	\specialrule{.1em}{.1em}{0.1em}
	time (s)  & 0.59  & 0.64 & 0.81 & 0.86 & 1.28  \\
	$f^{spa}_k$ & -3.3721 & -3.8006 & -2.4143 & -0.0689 & -0.0040 \\
	\specialrule{.2em}{0em}{0.1em}
\end{tabular}
\end{table}
As the relaxation order $k$ increases, we get better and better lower bounds. 
We are not sure if the sparse Moment-SOS hierarchy of \reff{eq:spcopmom}-\reff{eq:spcopsos} converges or not for this problem.
In comparison, the dense Moment-SOS relaxation of order $k=1$
is tight. The minimum value $f_{min} = 0$ and the global minimizer is 
$x^* = (3,\,0,\,0,\,2,\,0).$
\end{Example}

\begin{Example}  \label{ex:cubic_agm_6}  \rm 
Consider the sparse copositive optimization problem
\[
\left\{
\baray{cl}
\min & f(x) = f_1(x_{\Dt_1}) + f_2(x_{\Dt_2}) \\
\st & x_1 + 2x_2 + 3x_3 + x_4 + 2x_5 + 3x_6 = 2, \\
& 5x_1 - 2x_2 - 3x_3 + 4x_4 - x_5 - 3x_6 \ge 0, \\
& x_1 + x_2 + \cdots + x_6 = 1, \,x\in\re_+^6,
\earay
\right.
\]
where $\Dt_1 = \{1,2,3\}$, $\Dt_2 = \{4,5,6\}$ and
\[
\begin{array}{cl}
	f_1(x_1, x_2, x_3) &= x_1^2x_2 + x_1x_2^2 + x_3^3 - 3x_1x_2x_3, \\
	f_2(x_4, x_5, x_6) &= x_4^2x_5 + x_4x_5^2 + x_6^3 - 3x_4x_5x_6.
\end{array}
\]
The performance of the sparse Moment-SOS relaxations \reff{eq:spcopmom}-\reff{eq:spcopsos} for $k=2,\ldots, 6$ are shown in Table \ref{ta:order5}. 
In comparison, we solve this problem by the dense Moment-SOS relaxation with order 2, we get $f_{min} \approx 0.0000$ 
and the obtained global minimizer is
\[ 
x^* \, \approx \, (0.1550,\,   0.1544 ,\,   0.1550  ,\,0.1788  ,\,0.1780  ,\,  0.1788).
\] 
\begin{table}[htbp]
\caption{Computational results for Example~\ref{ex:cubic_agm_6}.}
\label{ta:order5}
\vspace{.1 in} 
\centering
\begin{tabular}{lccccc}
	\specialrule{.2em}{0em}{0.1em}
	$k$ & 2 & 3 & 4 & 5 & 6 \\
	\specialrule{.1em}{.1em}{0.1em}
	time (s)  & 0.72      & 0.70    & 1.92    & 8.26                 & 102.30 \\
	$f^{spa}_k$ & -344.1471 & -0.6765 & -0.0033 & $-5.5853\cdot 10^{-5}$ & $-8.6066\cdot 10^{-6}$ \\
	\specialrule{.2em}{0em}{0.1em}
\end{tabular}
\end{table}
\end{Example}

\begin{Example}\rm\label{ex:star_graph_7}
Consider the sparse copositive optimization problem
\[
\left\{
\baray{cl}
\min & f(x) = f_{1}(x_{\Dt_1}) + f_{2}(x_{\Dt_2}) + \cdots +  f_{6}(x_{\Dt_6})  \\
\st & 2x_1 + x_2 + \cdots + x_7 = 3,\, x\in\re_+^7,
\earay
\right.
\]
where $\Dt_1=\{1,2\}$, $\Dt_i = \{1, i+1\}$ for $i=2,\dots,6$ and
\begin{eqnarray*}
	f_{1}(x_1, x_2) &=&x_1^4 x_2^2 + x_1^2 x_2^4 - 3 x_1^3 x_2^3,  \\
	f_{i}(x_1, x_{i+1}) &=& x_{i+1}^6 + x_1^6,  \quad i=2,\ldots, 6.
\end{eqnarray*} 
By solving the sparse Moment-SOS relaxations \reff{eq:spcopmom}-\reff{eq:spcopsos} with $k=3$, we get the lower bound $f^{spa}_3 \approx 1.0013\cdot 10^{-5}$. 
By solving the dense Moment-SOS relaxations with order $k=3$, we get the minimum $f_{min} \approx 1.9641\cdot 10^{-6}$ and the global minimizer is
\[
x^* \,\approx \, (0.0000,\,    2.5739 ,\,   0.0856  ,\,  0.0854 ,\,   0.0850 ,\,   0.0854  ,\,  0.0846).
\]
For comparison, we apply traditional nonlinear optimization methods to solve this problem. For convenience of implementation, the {\tt MATLAB} function \texttt{fmincon} is used. We use the default parameter settings but consider different feasible initial points.
Some typical initial points are selected as follows:
\begin{align*}
x^{(1)} &= \big(\frac{3}{8}, \frac{3}{8}, \frac{3}{8}, \frac{3}{8}, \frac{3}{8}, \frac{3}{8}, \frac{3}{8}\big),\quad
x^{(2)} = \big(\frac{1}{2}, \frac{1}{2}, \frac{3}{10}, \frac{3}{10}, \frac{3}{10}, \frac{3}{10}, \frac{3}{10} \big), \\
x^{(3)} &= \big( \frac{6}{5}, \frac{3}{5}, 0, 0, 0, 0, 0 \big),\quad
x^{(4)} = ( 0, 3, 0, 0, 0, 0, 0 ), \\
x^{(5)} &= \left(0.1230, 1.4500, 0.2000, 0.4000, 0.1500, 0.2500, 0.1810\right), \\
x^{(6)} &= \left(0.5711, 0.5363, 0.2465, 0.3176, 0.5114, 0.1440, 0.1020\right),\\
x^{(7)} &= \left(0.1838, 0.7275, 0.4026, 0.4948, 0.5011, 0.4568, 0.0496\right), \\
x^{(8)} &= \left(0.0000,\,2.5000,\,0.0800,\,0.0800,\,0.0800,\,0.0800,\,0.0800 \right).
\end{align*}
For each instance, {\tt fmincon} fails to find the minimum value $f_{min}$, and it returns the objective value around $0.0061$.
\end{Example}

\begin{Example} \rm \label{ex:large_scale_random}
Consider randomly generated sparse copositive optimization problems of the form
\be  \label{sparse:qcqp}
\left\{ \baray{cll}
\min &   \sum\limits _{i=1}^m \big( x_{\Dt_i} ^T Q_i  x_{\Dt_i}+\sum\limits_{j\in\Dt_i}x_j x_{\Dt_i}^T A_i^j x_{\Dt_i}\big) \\
\st  & Ax = b,\, Cx\ge d,\, x\in\re_+^n.
\earay \right.
\ee
We explore the performance of sparse Moment-SOS relaxations \reff{eq:spcopmom}-\reff{eq:spcopsos} for different $m,n$.
\begin{itemize}
\item 	For cases that $m=n$, we choose the blocks $\Dt_i$ as
\[
\Dt_i = \left\{ \begin{array}{ll}
	\{ i,\ldots,i+w-1 \} & \mbox{for}\quad  i\le n-w+1,\\
	\{ i,\ldots,n,1,\ldots, w-n+i-1 \} & \mbox{for} \quad  i> n-w+1.
\end{array}
\right.
\]
\item For cases that $n$ is even and $m=n/2$, we choose the blocks $\Dt_i$ as
\[
\Dt_i = \left\{ \begin{array}{ll}
	\{ 2i-1,\ldots,2i-1+w-1 \} & \mbox{for} \quad 2i-1\le n-w+1,\\
	\{ 2i-1,\ldots,n,1,\ldots, 2i-1+w-n-1 \} & \mbox{for} \quad 2i-1> n-w+1.
\end{array}
\right.
\]
\end{itemize}
Entries of $A$ and $C$ are randomly generated obeying
independent and identically distributed standard Gaussian distributions.
To ensure the feasible set is nonempty, we construct vectors
$b$ and $d$ by setting a randomly generated point $x_0$ as a priori feasible point. In the objective, we let $Q_i = R_i^TR_i$ where each entry of $R_i$ 
obeys the standard Gaussian distribution. The matrix $A_i^j$ is generated in the same manner.

The feasible set of \eqref{sparse:qcqp} is a nonnegative polyhedron
and all $Q_i, A_i^j$ are positive semidefinite matrices.
Thus, each $f_i$ in \eqref{sparse:qcqp} is cop-SOS convex.
By Theorem~\ref{thm:cop-sos-cvx}, the sparse relaxations \reff{eq:spcopmom}-\reff{eq:spcopsos} 
are tight for all orders $k \ge k_0=2$.
To show the efficiency of the sparse Moment-SOS relaxations \reff{eq:spcopmom}-\reff{eq:spcopsos}, 
we compare it with the standard (dense) Moment-SOS relaxations  (see \reff{eq:pop_sos_k}).
The computational results are reported in Table~\ref{tab:rand}.
The notation $\operatorname{dim}(y_{spa})$, $\operatorname{dim}(y_{den})$ 
stand for the numbers of variables for the $2$nd order sparse and the dense moment relaxations, respectively. For distinct $n,\,m,\,w$, we use $t_{spa}$ to denote the average time for solving
\eqref{sparse:qcqp} with the sparse moment relaxations and use $t_{den}$ to denote the average time for the dense moment relaxations. The notation ``oom'' means the computer is out-of-memory.
%%%%%%%%%%%%%%%%%%%%%%%%%%%%%%%%%%%%%
\begin{table}[htbp]
\caption{Comparison between the sparse Moment-SOS relaxations \reff{eq:spcopmom}-\reff{eq:spcopsos} 
	and the dense relaxation for solving \reff{sparse:qcqp}.}
\label{tab:rand}
\vspace{.1in}
\centering
\begin{tabular}{ccccccc}
\specialrule{.2em}{0em}{0.1em}
\multirow{2}{*}{$n$} & \multirow{2}{*}{$m$} & \multicolumn{3}{c}{Sparse Relaxation} & \multicolumn{2}{c}{Dense Relaxation} \\
\cmidrule(lr){3-5} \cmidrule(lr){6-7}
&  & $w$ & $\operatorname{dim}(y_{spa})$ & $t_{spa}$ (s) & $\operatorname{dim}(y_{den})$ & $t_{den}$ (s) \\
\specialrule{.1em}{.1em}{0.1em}
20  & 20 & 3 & 401  & 0.64 & 10626   & 2085.07 \\
20  & 10 & 5 & 911  & 0.73 & 10626   & 2546.59 \\
20  & 10 & 6 & 1401 & 0.76 & 10626   & 2140.23 \\
30  & 30 & 4 & 1051 & 0.87 & 46376   & \text{oom} \\
30  & 15 & 6 & 2101 & 0.95 & 46376   & \text{oom} \\
30  & 15 & 8 & 4276 & 2.22 & 46376   & \text{oom} \\
50  & 50 & 5 & 2801 & 1.17 & 316251  & \text{oom} \\
50  & 25 & 6 & 3501 & 1.13 & 316251  & \text{oom} \\
50  & 25 & 8 & 7126 & 2.55 & 316251  & \text{oom} \\
80  & 80 & 3 & 801  & 2.21 & 1929501 & \text{oom} \\
80  & 40 & 5 & 3641 & 3.18 & 1929501 & \text{oom} \\
80  & 40 & 6 & 5601 & 2.58 & 1929501 & \text{oom} \\
100 & 50 & 3 & 1501 & 1.76 & 4598126 & \text{oom} \\
100 & 50 & 4 & 2751 & 2.27 & 4598126 & \text{oom} \\
100 & 50 & 5 & 4551 & 1.50 & 4598126 & \text{oom} \\
\specialrule{.2em}{0em}{0.1em}
\end{tabular}
\end{table}
% I moved this paragraph below the table
As shown in Table \ref{tab:rand}, it is much more efficient to solve \eqref{sparse:qcqp}  by the sparse Moment-SOS hierarchy 
than by the dense one. As the dimension $n$ increases,
the dense relaxation quickly becomes intractable.
In contrast, the sparse relaxations can solve larger sized problems.
\end{Example}

In the following, we show applications of sparse Moment-SOS relaxations for checking copositive tensors. A $d$th order $n$-dimensional tensor $\mc{A}$ can be labelled as the multi-dimensional array such that
\be    \label{eq:Apoly}
\mc{A} \coloneqq (\mc{A}_{i_1\ldots, i_d})_{1\le i_1,\ldots, i_d\le n}.
\ee
The tensor $\mc{A}$ is symmetric if its entries are invariant
under permutation of indices, i.e., 
$\mc{A}_{i_1 \ldots i_d} = \mc{A}_{j_1\ldots j_d}$ whenever 
$(i_1, \ldots, i_d)$ is a permutation of $(j_1, \ldots, j_d)$.
Let $\tt{S}^d(\re^n)$ denote the space of $d$th order $n$-dimensional
real symmetric tensors. Every $\mc{A}\in\mathtt{S}^d(\re^n)$ 
is uniquely determined by the homogeneous polynomial
\begin{equation}\label{eq:ten2poly}
	\mc{A}(x)\coloneqq \sum\limits_{1\le i_1,\ldots, i_d\le n}
	\mc{A}_{i_1,i_2\ldots i_d} x_{i_1}x_{i_2}\cdots x_{i_d}.
\end{equation}
The tensor $\mc{A}$ is said to be {\it copositive} if $\mc{A}(x)$ is copositive, 
i.e., $\mc{A}(x) \ge 0$ for all $x \in \re_+^n$.
To check if $\mc{A}$ is copositive or not, we can solve the
polynomial optimization problem:
\begin{equation}   \label{eq:tensorcop}
	\left\{\begin{array}{cl}
		\min  & \mc{A}(x)\\
		\st &  x_1+ \cdots + x_n = 1,\, x\in\re_n^+.
	\end{array}
	\right.
\end{equation}
Clearly, $\mc{A}$ is copositive if and only if the optimal value of the above is nonnegative. When $\mc{A}(x)$ has the sparse pattern like \eqref{eq:spa_pop}, the sparse Moment-SOS relaxations \reff{eq:spcopmom} and \reff{eq:spcopsos} can be applied to solve \reff{eq:tensorcop}.

\begin{Example} \rm
Consider the symmetric tensor $\mA \in \mathtt{S}^4(\re^4)$ such that
\[
\begin{array}{l}
\mA_{1111}=\mA_{4444}=1,\quad \mA_{2222}=\mA_{3333}=2,\\
\mA_{1122}=\mA_{1212}=\mA_{1221}=\mA_{2112}=\mA_{2121}=\mA_{2211}=-\frac{1}{6}, \\
\mA_{2233}=\mA_{2323}=\mA_{2332}=\mA_{3223}=\mA_{3232}=\mA_{3322}=\frac{1}{6}, \\
\mA_{3344}=\mA_{3434}=\mA_{3443}=\mA_{4334}=\mA_{4343}=\mA_{4433}=-\frac{1}{6},
\end{array}
\]
and all other entries are zeros.
The polynomial $\mc{A}(x)$ as in \eqref{eq:ten2poly} is 
\[
\mc{A}(x) = \underbrace{x_1^4+x_2^4-x_1^2x_2^2}_{f_1(x_{\Dt_1})}
+ \underbrace{x_2^4+x_3^4+x_2^2x_3^2}_{f_2(x_{\Dt_2})}
+ \underbrace{x_3^4+x_4^4-x_3^2x_4^2}_{f_3(x_{\Dt_3})},
\]
with the sparsity pattern sets 
\[
\Dt_{1} = \{1,2\},\quad \Dt_2 = \{2 ,3\},\quad \Dt_3 = \{3 ,4\}.
\]
We solve the copositive optimization problem \eqref{eq:tensorcop} by
the sparse Moment-SOS relaxations \reff{eq:spcopmom} and \reff{eq:spcopsos}.
For the lowest relaxation order $k = 2$, the obtained optimizer $y^*$
satisfies the rank condition
\[ 
\rank\, M_{\Dt_1}^{(2)}[y_{\Dt_1}^*] =
\rank\, M_{\Dt_2}^{(2)}[y_{\Dt_2}^*] =
\rank\, M_{\Dt_3}^{(2)}[y_{\Dt_3}^*] = 1.
\]
By Theorem~\ref{thm:rank1_tight}, the sparse relaxations \reff{eq:spcopmom} and \reff{eq:spcopsos} are tight.
We get the optimal value 
$
f_{min} = f^{spa}_2 \approx 0.0164,
$
and the obtained global optimizer is
\[
x^* = \pi(y*)  \,\approx\,  ( 0.2843,\, 0.2157,\, 0.2157,\, 0.2843).
\]
The computation took around $0.28$ seconds. The optimal value of \eqref{eq:tensorcop} is positive, so $\mA$ is copositive.
\end{Example}

\section{Conclusions}
\label{sec:con}

In this paper, we propose the sparse Moment-SOS relaxations \reff{eq:spcopmom} and \reff{eq:spcopsos} to solve the sparse copositive polynomial optimization problem \eqref{eq:spa_pop}. We give conditions to characterize tightness of these sparse Moment-SOS relaxations. We further show that these sparse Moment-SOS relaxations are tight under the cop-SOS convexity assumption. Large scale problems with the sparsity pattern can be solved by these sparse Moment-SOS relaxations. Numerical experiments are presented to show the performance.

\bigskip \noindent  
{\bf Acknowledgements}
Jiawang Nie is partially supported by the NSF grant DMS-2513254
and the AFOSR grant FA9550-25-1-0298.

%References


\begin{thebibliography}{plain}


\bibitem{AhmedDurStill13}
Ahmed, F., Dür, M. and Still, G.: Copositive programming via semi-infinite optimization. J. Optim. Theory Appl. 159, 322–340 (2013).


\bibitem{mosek}
ApS., M.:
Mosek optimization toolbox for MATLAB: User’s Guide and Reference Manual. Version 4.1, (2019).


\bibitem{BelKla02}
Belousov, E.G. and Klatte, D.: A Frank–Wolfe type theorem for convex polynomial programs. Comput. Optim. Appl. 22(1), 37-48 (2002).


\bibitem{DPB}
Bertsekas, D.P.: Nonlinear Programming: Second Edition. Athena Scientific, Belmont (2023).


\bibitem{BomzeDur00}
Bomze, I.M., Dür, M., De Klerk, E., Roos, C., Quist, A.J. and Terlaky, T.: On copositive programming and standard quadratic optimization problems. J. Glob. Optim. 18, 301–320 (2000).

\bibitem{Burer09}
Burer, S.: On the copositive representation of binary and continuous nonconvex quadratic programs. Math. Program. 120, 479–495 (2009).

\bibitem{Burer15}
Burer, S.: A gentle, geometric introduction to copositive optimization. Math. Program. 151, 89–116 (2015).


\bibitem{DeKlerk11}
De Klerk, E. and Laurent, M.: 
On the Lasserre hierarchy of semidefinite programming relaxations of
convex polynomial optimization problems. SIAM J. Optim. 21(3), 824–832 (2011).


\bibitem{DickinsonGijben14}
Dickinson, P. and Gijben, L.: On the computational complexity of membership problems for the completely positive cone and its dual. Comput. Optim. Appl. 57, 403–415 (2014).

\bibitem{DobreVera15}
Dobre, C. and Vera, J.: Exploiting symmetry in copositive programs via semidefinite hierarchies. Math. Program. 151, 659–680 (2015).

\bibitem{Dukanovic10}
Dukanovic, I. and Rendl, F.: Copositive programming motivated bounds on the stability and the chromatic numbers. Math. Program. 121, 249–268 (2010).

\bibitem{Dur10}
Dür, M.: Copositive programming – a survey. In: Diehl, M., Glineur, F., Jarlebring, E. and Michiels, W. (eds.): Recent advances in optimization and its applications in engineering, pp. 3–20. Springer, Berlin (2010).

\bibitem{EichfelderJahn08}
Eichfelder, G. and Jahn, J.: Set-semidefinite optimization. J. Convex Anal. 15(4), 767–801 (2008).

\bibitem{EichfelderPovh13}
Eichfelder, G. and Povh, J.: On the set-semidefinite representation of nonconvex quadratic programs over arbitrary feasible sets. Optim. Lett. 7, 1373–1386 (2013).

\bibitem{henrion2005detecting}
Henrion, D. and Lasserre, J.B.: Detecting global optimality and extracting solutions in GloptiPoly. In: Henrion, D. and Garulli, A. (eds.): Positive polynomials in control, pp. 293–310. Springer, Berlin (2005).

\bibitem{HuangKang25}
Huang, L., Kang, S., Wang, J. and Yang, H.: Sparse polynomial optimization with unbounded sets. SIAM J. Optim. 35, 593–621 (2025).

\bibitem{HuangXie25}
Huang, L. and Xie, L.: A finite-termination algorithm for testing copositivity over the positive semidefinite cone.
arXiv preprint \url{arXiv:2601.06648} (2025).

\bibitem{Kim2011}
Kim, S., Kojima, M., Mevissen, M. and Yamashita, M.: Exploiting sparsity in linear and nonlinear matrix inequalities via positive semidefinite matrix completion. Math. Program. 129, 33–68 (2011).

\bibitem{korda2025convergence}
Korda, M., Magron, V. and Rios-Zertuche, R.: Convergence rates for sums-of-squares hierarchies with correlative sparsity. Math. Program. 209(1), 435–473 (2025).

\bibitem{KLMS23}
Korda, M., Laurent, M., Magron, V. and Steenkamp, A.: Exploiting ideal-sparsity in the generalized moment problem with application to matrix factorization ranks. Math. Program. 205(1), 703–744 (2024).


\bibitem{Las01}
Lasserre, J.B.: Global optimization with polynomials and the problem of moments. SIAM J. Optim. 11, 796–817 (2000).

\bibitem{Lasspacov06}
Lasserre, J.B.: Convergent SDP-relaxations in polynomial optimization with sparsity. SIAM J. Optim. 17, 822–843 (2006).



\bibitem{Las09}
Lasserre, J.B.: Convexity in semi-algebraic geometry and polynomial optimization. SIAM J. Optim. 19, 1995–2014 (2009).


\bibitem{Lasserre14}
Lasserre, J.B.: New approximations for the cone of copositive matrices and its dual. Math. Program. 144, 265–276 (2014).


\bibitem{LasBk15}
Lasserre, J.B.: An Introduction to Polynomial and Semi-Algebraic Optimization. Vol. 52. Cambridge University Press, Cambridge (2015).


\bibitem{Lau09}
Laurent, M.: Sums of squares, moment matrices and optimization over polynomials. In: Putinar, M. and Saff, E.B. (eds.): Emerging applications of algebraic geometry, pp. 157–270. Springer, Berlin (2009).

\bibitem{yalmip}
Löfberg, J.: YALMIP: A toolbox for modeling and optimization in MATLAB. In: IEEE International Conference on Robotics and Automation, pp. 284–289. IEEE, Piscataway (2004).

\bibitem{MagronWang23}
Magron, V. and Wang, J.: Sparse Polynomial Optimization: Theory and Practice. World Scientific, Singapore (2023).

\bibitem{vmjw21}
Magron, V. and Wang, J.: TSSOS: A Julia library to exploit sparsity for large-scale polynomial optimization. arXiv preprint \url{arXiv:2103.00915} (2021).


\bibitem{NieFlat13}
Nie, J.: Certifying convergence of Lasserre's hierarchy via flat truncation. Math. Program. 142, 485–510 (2013).


\bibitem{nieopcd}
Nie, J.: Optimality conditions and finite convergence of Lasserre's hierarchy. Math. Program. 146(1-2), 97–121 (2014).


\bibitem{NieBook}
Nie, J.: Polynomial and Moment Optimization. SIAM, Philadelphia (2023).

\bibitem{NieDemmel09}
Nie, J. and Demmel, J.: Sparse SOS relaxations for minimizing functions that are summations of small polynomials. SIAM J. Optim. 19(4), 1534–1558 (2009).


\bibitem{NieYangZhang18}
Nie, J., Yang, Z. and Zhang, X.: A complete semidefinite algorithm for detecting copositive matrices and tensors. SIAM J. Optim. 28(4), 2902–2922 (2018).

\bibitem{NieDehom23}
Nie, J., Tang, X., Yang, Z. and Zhong, S.: Dehomogenization for completely positive tensors. Numer. Algebra Control Optim. 13(2), 340–363 (2023).


\bibitem{NieQuTangZhang25}
Nie, J., Qu, Z., Tang, X. and Zhang, L.: A characterization for tightness of the sparse moment-SOS hierarchy. Math. Program. 215, 369–405 (2026).

\bibitem{NQTZ25mat}
Nie, J., Qu, Z., Tang, X. and Zhang, L.: Sparse polynomial optimization with matrix constraints. J Glob Optim (2026). \url{https://doi.org/10.1007/}


\bibitem{PovhRendl07}
Povh, J. and Rendl, F.: A copositive programming approach to graph partitioning. SIAM J. Optim. 18(1), 223–241 (2007).

\bibitem{PowersReznick01}
Powers, V. and Reznick, B.: A new bound for Pólya’s theorem with applications to polynomials positive on polyhedral. J. Pure Appl. Algebra 164, 221–229 (2001).

\bibitem{Putinar93}
Putinar, M.: Positive polynomials on compact semi-algebraic sets. Indiana Univ. Math. J. 42, 969–984 (1993).

\bibitem{QuTang24}
Qu, Z. and Tang, X.: A correlatively sparse Lagrange multiplier expression relaxation for polynomial optimization. SIAM J. Optim. 34(1), 127–162 (2024).


\bibitem{VargasLaurent23}
Vargas, L.F. and Laurent, M.: Copositive matrices, sums of squares and the stability number of a graph. In: Kojima, M., Lasserre, J.B. and Takeda, A. (eds.): Polynomial optimization, moments, and applications, pp. 113–151. Springer, Cham (2023).

\bibitem{waki2006sums}
Waki, H., Kim, S., Kojima, M. and Muramatsu, M.: Sums of squares and semidefinite program relaxations for polynomial optimization problems with structured sparsity. SIAM J. Optim. 17(1), 218–242 (2006).

\bibitem{TSSOS}
Wang, J., Magron, V. and Lasserre, J.B.: TSSOS: A Moment-SOS hierarchy that exploits term sparsity. SIAM J. Optim. 31(1), 30–58 (2021).

\bibitem{CSTSSOS}
Wang, J., Magron, V., Lasserre, J.B. and Mai, N.H.A.: CS-TSSOS: Correlative and term sparsity for large-scale polynomial optimization. ACM Trans. Math. Software 48(4), 1–26 (2022).

\end{thebibliography}
\end{document}